\documentclass[conference]{IEEEtran}
\IEEEoverridecommandlockouts

\usepackage[T1]{fontenc}
\usepackage[latin9]{inputenc}
\usepackage[compress]{natbib}
\usepackage{amsthm}
\usepackage{amsmath}
\usepackage{mathdots}
\usepackage{algcompatible}
\usepackage{algorithm}
\usepackage{booktabs}
\usepackage{enumerate}
\usepackage{graphicx}
\usepackage{amssymb}
\usepackage{latexsym}
\usepackage{epstopdf}
\usepackage{color}
\usepackage{bbm}
\usepackage{needspace} 
\usepackage{color, colortbl}
\usepackage[english]{babel}
\usepackage{tikz}
\usetikzlibrary{calc,fit,matrix,arrows,automata,positioning}

\usepackage{caption}
\usepackage{subcaption}
\usetikzlibrary{arrows,automata}
\usetikzlibrary{positioning}
\usepackage{filecontents}
\setcitestyle{numbers}
\DeclareCaptionFont{mysize}{\fontsize{8}{9.6}\selectfont}
\ifCLASSINFOpdf
\else
\fi
\setcitestyle{square}
\theoremstyle{plain}
\newtheorem{thm}{Theorem}
\theoremstyle{plain}
\newtheorem{lem}{Lemma}
\newtheorem{prop}{Proposition}
\newtheorem{remark}{Remark}
\newtheorem{coro}{Corollary}

\newtheorem{define}{Definition}
\newtheorem{problem}{Problem}

\newtheorem{assumption}{Assumption}

\makeatother

\usepackage{babel}
\newcommand{\norm}[1]{\left\lVert#1\right\rVert}

\graphicspath{D:/Dropbox/{Structural_controllability_of_undirected_network}/CDC}

\setcitestyle{citesep={,}}

\newcommand{\tikzmark}[1]{\tikz[overlay,remember picture] \node (#1) {};}
\newcommand{\tikzdrawbox}[3][(0pt,0pt)]{%
	\tikz[overlay,remember picture]{
		\draw[#3]
		($(left#2)+(-1.8em,1em) + #1$) rectangle
		($(right#2)+(1.6em,-1em) - #1$);}
}
\begin{document}

\addtolength{\textheight}{0cm}
\addtolength{\voffset}{0.35in}
\title{Structural Target Controllability of Undirected Networks}
\author{Jingqi Li, Ximing Chen, S\'{e}rgio Pequito, George J. Pappas, and Victor M. Preciado%
\thanks{Jingqi Li, Ximing Chen, George J. Pappas and Victor M. Preciado are with the Department of Electrical and Systems Engineering
at the University of Pennsylvania, Philadelphia, PA 19104.  e-mail: \{jingqili,ximingch,pappasg,preciado\}@seas.upenn.edu. S{\'e}rgio Pequito is with the Department of Industrial and Systems Engineering at the Rensselaer Polytechnic Institute, Troy, NY 12180-3590. email: goncas@rpi.edu.%
}}
\maketitle
\begin{abstract}
In this paper, we study the \emph{target controllability problem} of networked dynamical systems, in which we are tasked to steer a subset of network states towards a desired objective.
More specifically, we derive necessary and sufficient conditions for the \emph{structural target controllability problem} of linear time-invariant (LTI) systems with symmetric state matrices, such as undirected dynamical networks with unknown link weights.
To achieve our goal, we first characterize the generic rank of \emph{symmetrically structured matrices}, as well as the modes of any numerical realization.
Subsequently, we provide a graph-theoretic necessary and sufficient condition for the structural controllability of undirected networks with multiple control nodes, for which there are no previous results in the literature.
Finally, we derive a graph-theoretic necessary and sufficient condition for structural target controllability of undirected networks. Remarkably, apart from the standard reachability condition, only local topological information is needed for the verification of structural target controllability.
\end{abstract}
	
\section{Introduction}\label{sec:intro}
Complex networks have been shown to be a powerful tool for modeling dynamical systems~\cite{liu2011controllability,cowan2012nodal,pasqualetti2014controllability}. In particular, when analyzing and designing networked systems, it is crucial to verify their controllability, i.e., a property ensuring the existence of an input sequence allowing us to drive the states of the system towards arbitrary states within finite time. Nonetheless, verifying such a property requires full knowledge of the parameters describing the system's dynamics~\cite{kalman1963mathematical}. In applications involving large-scale networks, those parameters are difficult, or even impossible, to obtain~\cite{barco2009automatic}. Alternatively, it is more viable to identify the presence of dynamical interconnections among the states of a network. Subsequently, it is of interest to analyze system properties such as controllability using topological information of the system dynamics, which led to the development of system analysis tools using graph theory \cite{murota2012systems}.

Seminal work on graph-theoretic analysis of controllability can be found in~\cite{lin1974structural}, in which the notion of \emph{structural controllability} was stated. Following this seminal work, the authors in \cite{shields1976structural,glover1976characterization,hosoe1979irreducibility,hosoe1980determination} provided necessary and sufficient conditions for structural controllability of multi-input linear time-invariant (LTI) systems using various graph-theoretic notions. Nonetheless, existing results on structural controllability assumed implicitly that the parameters are either fixed zeros or independent free variables. Such an assumption is often violated in practical scenarios, for instance, when the system is characterized by undirected networks \cite{myers2014information}, or when different interconnections in the system are strongly correlated \cite{pagani2013power}. Consequently, it is of interest to provide necessary and sufficient conditions for structural systems characterized by graph with special weight constraints. Such problems are considered in \cite{corfmat1976structurally} and \cite{anderson1982structural}. However, the result in \cite{corfmat1976structurally} is not applicable to systems modeled by undirected graph, whereas the matrix net approach in \cite{anderson1982structural} may suffer from computational complexity in large-scale systems. Recently, the authors in \cite{menara2017structural} and \cite{mousavi2017structural} proposed graph-theoretic necessary and sufficient conditions for structural controllability of dynamical systems modeled by symmetric graph. Different from their approaches, in this paper, we provide a full characterization of the controllable modes using structural information of an undirected network, which facilitates a deeper understanding of structural controllability for systems involving symmetric parameter constraints.

However, in certain scenarios, we are only concerned about our ability to steer a collection of states, which can be captured by the notion of \emph{(structural) target controllability} \cite{gao2014target,czeizler2018structural}. The target controllability problem is a particular case of output controllability problem \cite{petre2016target}, while the necessary and sufficient condition of structural output controllability is still unknown \cite{murota1990note}. 
Recently, the authors in \cite{monshizadeh2015strong,van2017distance} proposed conditions for strong target controllability using zero-forcing sets. Nonetheless, to the best of our knowledge, providing necessary and sufficient conditions for structural target controllability 
remains an unsolved problem.

In this paper, we derive necessary and sufficient conditions for the problem of structural target controllability of LTI systems with symmetric state matrices, such as undirected dynamical networks with unknown link weights. Our contribution is three-fold: First, we introduce the concept of symmetrically structured matrix. We then characterize the generic rank of symmetrically structured matrices, as well as generic spectral properties of any numerical realization. Secondly, we propose graph-theoretic necessary and sufficient conditions for structural controllability of undirected networks with multiple inputs. Finally, we derive a necessary and sufficient condition for structural target controllability of undirected networks.


The rest of the paper is organized as follows. In Section~\ref{sec:prelim}, we introduce preliminaries in algebra and graph theory. We formulate the problem under consideration in Section~\ref{sec:problem}. In Section~\ref{sec:main}, we present our main results. 
In Section~\ref{sec:example}, we present an example to illustrate our results. Finally, we conclude the paper in Section~\ref{sec:con}. All the proofs are included in the Appendix.

\addtolength{\textheight}{0cm}
\addtolength{\voffset}{-0.25in}
\section{Notation and Preliminaries}\label{sec:prelim}
We denote the cardinality of a set $\mathcal{S}$
by $\left|\mathcal{S}\right|$. We adopt the notation $[n]$ to represent the set of integers $\{1,\ldots,n\}.$ Let $\mathbf{0}_{n\times m}\in\mathbb{R}^{n\times m}$ be the matrix with all entries equals to zero. Whenever clear from the context, $\mathbf{0}_{n\times m}$ is abbreviated as~$\mathbf{0}$. 

Given $M_1\in\mathbb{R}^{n\times m_1}$ and $M_2\in\mathbb{R}^{n\times m_2}$, we let $[M_1,M_2] \in \mathbb{R}^{n\times (m_1 + m_2)}$ be the concatenation of $M_1$ and $M_2$. The $ij$-th entry of $M\in\mathbb{R}^{n\times n}$ is denoted by $[M]_{ij}.$ Moreover, we let $[M]_{i_1,\dots,i_k}^{j_1,\dots, j_k}$ be the $k\times k$ submatrix of $M$ formed by collecting $i_1,\dots,i_k$-th rows and $j_1,\dots,j_k$-th columns of $M$. The determinant of a matrix $M\in \mathbb{R}^{n\times n}$ is defined by the expansion: 
	\begin{equation}\label{eq:detexpand}\small
	\det M=\sum_{\sigma\in \mathcal{S}_n}\left(\textrm{sgn}(\sigma)\prod_{i=1}^n \left[M\right]_{i\sigma(i)}\right),
	\end{equation}
	where $\mathcal{S}_n$ is the set of all permutations of $\{1,\dots,n\}$, and $\text{sgn}(\sigma)$ is the \emph{signature}\footnote{The signature of a permutation equals to $1$ if $|\{(x,y): x<y, \sigma(x) > \sigma(y)\}|$ is even, and $-1$ otherwise.} of a permutation $\sigma\in \mathcal{S}_n$. 
	
A matrix $\bar{M}\in\{0,\star\}^{n\times m}$ is called a \emph{structured matrix}, if $[\bar{M}]_{ij}$ is either a fixed zero or an independent free parameter denoted by $\star.$ In particular, we define a matrix $\bar{M}\in\{0,\star\}^{n\times n}$ to be \emph{symmetrically structured}, if the value of the free parameter associated with $[\bar{M}]_{ij}$ is constrained to be the same as the value of the free parameter associated with $[\bar{M}]_{ji}$
, for all $j$ and $i$. For example, consider $\bar{M}$ and $\bar{A}$ be specified by
\begin{equation*}
	\bar{M}=\begin{bmatrix}
	0 & m_{12} \\
	m_{21} & m_{22}\\
	\end{bmatrix}\text{ and } \bar{A}=\begin{bmatrix}
	0& a_{12} \\
	a_{12} & a_{22}\\
	\end{bmatrix},
\end{equation*}
	where $m_{12},m_{21},m_{22}$ and $a_{12},a_{22}$ are independent parameters. In this case, $\bar{M}$ is a structured matrix whereas $\bar A$ is symmetrically structured.


In the rest of the paper, we refer to $\tilde{M}$ as a \emph{numerical realization} of a (symmetrically) structured matrix $\bar{M}$, i.e., $\tilde{M}$ is a matrix obtained by independently assigning real numbers to each independent free parameter in $\bar{M}$. {In addition, we say that the structured matrix $\bar{M}\in\{0,\star\}^{n\times m}$ is the \emph{structural pattern} of the matrix $M\in\mathbb{R}^{n\times m}$, where $[\bar{M}]_{ij}=\star$ if and only if $[M]_{ij}\ne 0$, for $\forall i\in[n], \forall j\in[m]$.} 
	
Given a (symmetrically) structured matrix $\bar{M}$, we let $n_{\bar{M}}$ be the number of its independent free parameters and we associate with $\bar{M}$ a parameter space $\mathbb{R}^{n_{\bar{M}}}$. Furthermore, we use vector $\mathbf{p}_{\tilde{M}}=(p_1,\dots,p_{n_{\bar{M}}})^\top\in\mathbb{R}^{n_{\bar{M}}}$ to encode the value of independent free entries of $\bar{M}$ in a numerical realization $\tilde{M}$. 

In what follows, a set $V\subseteq\mathbb{R}^n$ is called a \emph{variety} if there exist polynomials $\varphi_1,\dots,\varphi_k$, such that $V=\{x\in\mathbb{R}^n\colon \varphi_i(x)=0,\forall i\in\{1,\dots,k\}\}$, and $V$ is a \emph{proper variety} when~$V\ne \mathbb{R}^n$. We denote by $V^c:=\mathbb{R}^n\setminus V$ its complement.
	
	{
		The \emph{term rank} \cite{murota2012systems} of a (symmetrically) structured matrix $\bar{M}$, denoted as $\textrm{t--rank}(\bar{M})$, is the largest integer $k$ such that, for some suitably chosen distinct rows $i_1,\dots,i_k$ and distinct columns $j_1,\dots,j_k$, all of the entries $\{ [\bar{M}]_{i_\ell j_\ell} \}_{\ell = 1}^k$ are $\star$-entries. Additionally, a (symmetrically) structured matrix $\bar{M}\in\{0,\star\}^{n\times m}$ is said to have \emph{generic rank} $k$, denoted as $\textrm{g--rank}(\bar{M})=k$, if there exists a numerical realization $\tilde{M}$ of $\bar{M}$, such that $\textrm{rank}(\tilde{M})=k$. If $\textrm{g--rank}(\bar{M})>0$, it is worth noting that the set of parameters describing all possible realizations forms a proper variety when $\textrm{rank}(\tilde{M})<\textrm{g--rank}(\bar{M})$ \cite{hosoe1980determination}. 
		
	

In the rest of the paper, we let $\mathcal{D}=\left(\mathcal{V},\mathcal{E}\right)$ denote a directed graph whose vertex-set and edge-set are denoted by $\mathcal{V}=\{v_1,\ldots,v_n\}$ and $\mathcal{E}\subseteq \mathcal{V}\times\mathcal{V},$ respectively. A \emph{path} $\mathcal{P}$ in $\mathcal{D}$ is defined as an ordered sequence of distinct vertices $\mathcal{P}=(v_1,\ldots,v_k)$ with $\{v_1,\dots,v_k\}\subseteq \mathcal{V}$ and $(v_{i}, v_{i+1})\in \mathcal{E}$ for all $i=1,\ldots,k-1$. A \emph{cycle} is either a path $(v_1,\ldots,v_k)$ with an additional edge $(v_k,v_1)$ (denoted as $\mathcal{C}=(v_1,\dots,v_k,v_1)$), or a vertex with an edge to itself (i.e., self-loop, denoted as cycle $\mathcal{C}=(v_1,v_1)$). We denote by $\mathcal{V_{C}}\subseteq\mathcal{V}$ the set of vertices in $\mathcal{C}$, and $\mathcal{E_C}\subseteq\mathcal{E}$ the set of edges in $\mathcal{C}$. The length of a cycle $\mathcal{C}$, is defined as the number of distinct vertices in $\mathcal{C}$, and is denoted by $\left|\mathcal{C}\right|$. Given a set $\mathcal S$ of vertices in $\mathcal{D},$ we let $\mathcal{D}_{\mathcal S} = (\mathcal S,\mathcal S\times \mathcal S\subset \mathcal E)$ be the \emph{subgraph of $\mathcal{D}$ induced by} $\mathcal{S}.$ We say that $\mathcal{D}_{\mathcal S}$ can be covered by \emph{disjoint cycles} if there exists $\mathcal{C}_1,\dots ,\mathcal{C}_l$, such that $\mathcal{S}=\bigcup_{i=1}^{l}\mathcal{V}_{\mathcal{C}_i}$ and $\mathcal{V}_{\mathcal{C}_i}\cap\mathcal{V}_{\mathcal{C}_j}=\emptyset$, for all $i\ne j$, $i,j\in\{1,\dots,l\}$. Given a set $\mathcal{S}\subseteq\mathcal{V}$, we define the \emph{in-neighbour set} of $\mathcal{S}$ as $\mathcal{N(S)}=\{v_i\in\mathcal{V}|(v_i,v_j)\in\mathcal{E},v_j\in\mathcal{S}\}$. We say a vertex $v_i$ is \emph{reachable} from vertex $v_j$ in $\mathcal{D}(\mathcal{V,E})$, if there exists a path from vertex $v_j$ to vertex $v_i$.

	Given a directed graph $\mathcal{D} = (\mathcal{V}, \mathcal{E})$ and two sets $\mathcal{S}_1, \mathcal{S}_2\subseteq \mathcal{V}$, we define the \emph{bipartite graph} $\mathcal{B}(\mathcal{S}_1, \mathcal{S}_2, \mathcal{E}_{\mathcal{S}_1, \mathcal{S}_2})$ as an undirected graph, whose vertex set is $\mathcal{S}_1\cup \mathcal{S}_2$ and edge set\footnote{We denote undirected edges using curly brackets $\{v_i,v_j\}$, in contrast with directed edges, for which we use parenthesis.} $\mathcal{E}_{\mathcal{S}_1, \mathcal{S}_2} = \{\{s_1, s_2\}\colon(s_1,s_2)\in\mathcal{E},  s_1\in \mathcal{S}_1, s_2\in \mathcal{S}_2\}.$ Given $\mathcal{B}(\mathcal{S}_1, \mathcal{S}_2, \mathcal{E}_{\mathcal{S}_1, \mathcal{S}_2})$, and a set $\mathcal{S}\subseteq\mathcal{S}_1$ or $\mathcal{S}\subseteq\mathcal{S}_2$, we define \emph{bipartite neighbor set} of $\mathcal{S}$ as $\mathcal{N_B(S)}=\{j\colon\{j,i\}\in\mathcal{E}_{\mathcal{S}_1,\mathcal{S}_2},i\in\mathcal{S}\}$. A \emph{matching} $\mathcal{M}$ is a set of edges in $\mathcal{E}_{\mathcal{S}_1, \mathcal{S}_2}$ that do not share vertices, i.e., given edges $e = \{s_1, s_2\}$ and $e^\prime = \{s_1^\prime, s_2^\prime\}$, $e,e^\prime \in \mathcal{M}$ only if $s_1 \neq s_1^\prime$ and $s_2 \neq s_2^\prime.$ 
	The vertex $v$ is said to be \emph{right-unmatched} with respect to a matching $\mathcal{M}$ associated with $\mathcal{B}(\mathcal{S}_1, \mathcal{S}_2, \mathcal{E}_{\mathcal{S}_1, \mathcal{S}_2})$ if $v \in \mathcal{S}_2$, and $v$ does not belong to an edge in the matching $\mathcal{M}$.

\section{Problem Statements}\label{sec:problem}
We consider a linear time-invariant system whose dynamics is captured by
\begin{equation}\small
\dot{x}=Ax+Bu,\ \ \  y=Cx,
\label{eq:1}
\end{equation}
where $x\in\mathbb{R}^{n}$, $y\in\mathbb{R}^{k}$ and $u\in\mathbb{R}^{m}$ are the state, the output and the input vectors, respectively. In addition, the matrix $A\in\mathbb{R}^{n\times n}$ is the state matrix, $B\in\mathbb{R}^{n\times m}$ is the input matrix and $C\in\mathbb{R}^{k\times n}$ is the output matrix. In this paper, we consider the following assumption:

\begin{assumption}\label{assumption_sym}
The state matrix $A\in\mathbb{R}^{n\times n}$ is symmetric, i.e., $A=A^\top.$ 
\end{assumption}
This symmetry assumption is motivated by control problems arising in undirected networked dynamical systems. Furthermore, this assumption will be crucial when establishing graph-theoretic results characterizing structural controllability problems in undirected networks.

Hereafter, we use the 3-tuple $(A,B,C)$ to represent the system~\eqref{eq:1}. In particular, we use the pair $(A,B)$ to denote a system without a measured output. A pair $(A,B)$ is called \emph{reducible} if there exists a permutation matrix $P$, such that 
	\begin{equation}\small
	PAP^{-1}=\begin{bmatrix}
	A_{11}& \mathbf{0}\\
	A_{21}& A_{22}
	\end{bmatrix},\ \ PB=\begin{bmatrix}
	\mathbf{0} \\
	B_{2}
	\end{bmatrix},
	\end{equation} {where $A_{11}\in\mathbb{R}^{q\times q}$ and $B_{2}\in\mathbb{R}^{(n-q)\times m}$, $1\le q< n$}. The pair $(A,B)$ is called \emph{irreducible} otherwise. Furthermore, we use $\bar{A}$ and $\bar{B}$ to represent the structural pattern of $A$ and $B,$ respectively. In particular, by Assumption~\ref{assumption_sym}, we consider $\bar{A}$ to be symmetrically structured. Thus, $(\bar{A},\bar{B})$ is referred to as the \emph{structural pair} of the system $(A,B).$ 
	Given a structured matrix $\bar{A},$ we associate it with a directed graph $\mathcal{D}(\bar{A})=(\mathcal{X,E_{X,X}}),$ which we refer to as the \emph{state digraph}, where $\mathcal{X}=\{x_1,\dots,x_n\}$ is the state vertex set, and $\mathcal{E_{X,X}}=\{(x_j,x_i):[\bar{A}]_{ij} = \star \}$ is the set of edges. Similarly, we associate a directed graph $\mathcal{D}(\bar{A},\bar{B})=(\mathcal{X}\cup\mathcal{U}, \mathcal{E}_{\mathcal{X},\mathcal{X}}\cup \mathcal{E}_{\mathcal{U},\mathcal{X}})$ with the structural pair $(\bar{A},\bar{B})$, where $\mathcal{U}=\{u_1,\dots,u_m\}$ is the set of input vertices and $\mathcal{E}_{\mathcal{U},\mathcal{X}} = \{(u_j, x_i)\colon [\bar{B}]_{ij} =\star\}$ is the set of edges from input vertices to state vertices. We refer to $\mathcal{D}(\bar{A},\bar{B})$ as the \emph{system digraph}.

\begin{define}[Structural Controllability \cite{lin1974structural}]\label{structuralcontrollability}
A structural pair $(\bar{A},\bar{B})$ is structurally controllable if there exists a numerical realization $(\tilde{A},\tilde{B})$, such that {the controllability matrix $Q(\tilde{A},\tilde{B}):=[\tilde B,\tilde{A}\tilde{B},\dots,\tilde{A}^{n-1}\tilde{B}]$ has full row rank}.
\end{define}

While controllability is concerned about the ability to steer all the states of a system to a desired final state, under certain circumstances, it is more preferred to control the behavior of only a subset of states. More specifically, given a set $\mathcal{T} \subseteq [n]$, which we refer to as the \emph{target set}, it is of interest to consider whether the set of selected states can be steered arbitrarily. If so, we say that the pair $(A,B)$ is \emph{target controllable} with respect to $\mathcal{T}$~\cite{gao2014target}. Notice that this does not exclude the possibility of some other states indexed by $[n]\setminus \mathcal{T}$ being controllable as well. Similarly, we introduce the notion of \emph{structural target controllability} in the context of structural pairs.
 
\begin{define}[Structural Target Controllability \cite{czeizler2018structural}]
Given a structural pair $(\bar{A},\bar{B}),$ and a target set $\mathcal{T} = \{i_1,\dots,i_k\} \subseteq [n],$ let $\mathcal{X_T}$ be the set of state vertices corresponding to $\mathcal{T}$ in $\mathcal{D}(\bar{A},\bar{B})$. We define a matrix $C_{\mathcal{T}}\in\mathbb{R}^{k\times n}$ by
\begin{equation}\label{CT}\small
\left[C_{\mathcal{T}}\right]_{\ell j}= 
 \begin{cases}
 1, & \text{ if } j=i_\ell,\  i_\ell\in \mathcal T,\\
 0, & \text{otherwise.}
  \end{cases}
\end{equation}
The structural pair $(\bar{A},\bar{B})$ is structurally target controllable with respect to $\mathcal{T}$ if there exists a numerical realization $(\tilde{A},\tilde{B})$, such that the target controllability matrix $Q_{\mathcal{T}}(\tilde{A},\tilde{B}):=C_{\mathcal{T}}[\tilde{B},\tilde{A}\tilde{B},\dots,\tilde{A}^{n-1}\tilde{B}]$ has full row rank. 
\end{define}
Note that structural controllability is equivalent to structural target controllability when $\mathcal{T} = [n].$ Therefore, the necessary and sufficient conditions for structurally target controllable undirected networks can be applied to characterize structural controllability. Subsequently, in this paper, we consider the following problem:
\begin{problem}\label{prob}
Given a structural pair $(\bar{A},\bar{B})$, where $\bar{A}$ is symmetrically structured and $\bar{B}$ is a structured matrix, and a target set $\mathcal{T}\subseteq [n]$, find a necessary and sufficient condition for $(\bar{A},\bar{B})$ to be structurally target controllable with respect to~$\mathcal{T}.$
\end{problem}	

\section{Main Results}\label{sec:main}
In this section, we first introduce a proposition that is crucial for developing our solution to Problem~1. Then, we characterize the generic rank of symmetrically structured matrices in Lemma~\ref{symmetric}. Subsequently, we characterize the relationship between the term-rank of a symmetrically structured matrix and the presence of non-zero simple eigenvalues in a numerical realization in Lemma~\ref{lem4}. 
{This allows us to obtain a result characterizing the relationship between irreducibility and structural controllability of a structural pair involving symmetrically structured matrix (see Lemma~\ref{theo2}).} Based on these results, we propose graph-theoretic necessary and sufficient conditions for structural controllability and structural target controllability in Theorems~\ref{theo3} and~\ref{theo4}, respectively. 




	\begin{prop}[Popov-Belevitch-Hautus (PBH) test \cite{kailath1980linear}]\label{pbh}
	The pair $(A,B)$ is uncontrollable if and only if there exists a $\lambda\in\mathbb{C}$ and a nontrivial vector $e\in\mathbb{C}^n$, such that $e^\top A=\lambda e^\top $ and $e^\top B=0$.
\end{prop}
Given a pair $(A,B)$, where $A\in\mathbb{R}^{n\times n}$ and $B\in\mathbb{R}^{n\times m}$, we say that the mode $(\lambda,e)$ of $A$, where $\lambda\in\mathbb{C}$ and $e\in\mathbb{C}^n$, is an \emph{uncontrollable mode} if $e^\top A=\lambda e^\top$ and $e^\top B=0$.



\subsection{Generic Properties}\vspace{-0.1cm}
If a symmetrically structured matrix is generically full rank, then any numerical realization has almost surely no zero eigenvalue. 
In this subsection, we characterize the generic rank of a symmetrically structured matrix in Lemma~\ref{symmetric}, which lays the foundation for a further characterization of spectral properties of numerical realizations.
	
\begin{lem}\label{symmetric}
{Consider an $n\times n$ symmetrically structured matrix $\bar{A}$, and a set $\mathcal{T}=\{{i_1},\dots,{i_k}\}\subseteq [n]$. Let $\mathcal{D}(\bar{A})=(\mathcal{X},\mathcal{E_{X,X}})$ be the digraph representation of $\bar{A}$, $\mathcal{X_T}\subseteq \mathcal{X}$ be the set of vertices indexed by $\mathcal{T}$, and $C_{\mathcal{T}}$ be defined as in~\eqref{CT}. The generic-rank of $C_{\mathcal{T}}\bar{A}$ equals to $k$ if and only if $\mathcal{\left|N(S)\right|\ge \left|S\right|}$, $\forall \mathcal{S}\subseteq\mathcal{X_T}$.}
\end{lem}

Lemma~\ref{symmetric} establishes a relationship between the generic rank of a submatrix of a symmetrically structured matrix and the topology of its corresponding digraph. 
Subsequently, Corollary~\ref{coro} follows, which 
characterizes the generic rank of the concatenation of a symmetrically structured matrix $\bar{A}\in\{0,\star\}^{n\times n}$ and a structured matrix $\bar{B}\in\{0,\star\}^{n\times m}$.

\begin{coro}\label{coro}
Consider a structural pair $(\bar{A},\bar{B})$, where $\bar{A}$ is symmetrically structured, and a set $\mathcal{T}=\{{i_1},\dots,{i_k}\}\subseteq [n]$. Let $\mathcal{D}(\bar{A},\bar{B})=(\mathcal{X}\cup\mathcal{U}, \mathcal{E}_{\mathcal{X},\mathcal{X}}\cup \mathcal{E}_{\mathcal{U},\mathcal{X}}) $ be the digraph representation of $(\bar{A},\bar{B})$, and $\mathcal{X_T}\subseteq \mathcal{X}$ be the set of vertices indexed by $\mathcal{T}.$ If $\left|\mathcal{N(S)}\right|\ge \left|\mathcal{S}\right|$, $\forall\mathcal{S}\subseteq\mathcal{X_T}$, then $\emph{g--rank} (C_{\mathcal{T}}[\bar{A},\bar{B}]) = k.$
\end{coro}

In the remaining subsections, we aim to provide necessary and sufficient conditions for structural controllability. To achieve this goal, we notice that the eigenvalues of the state matrix are closely related to controllability, as indicated by Proposition~\ref{pbh}. 
{Besides, the approach in \cite{hosoe1979irreducibility} shows that for an irreducible structural pair with no symmetric parameter dependencies, all the nonzero modes of its numerical realization are almost surely simple and controllable. Similarly, to characterize structural controllability of undirected networks, we will provide characterizations of the modes in the numerical realization of a structural pair involving symmetrically structured matrix. Instead of using the maximum order of principle minor as in \cite{hosoe1979irreducibility}, 
we derive below a condition based on the term rank to ensure that generically the numerical realization of a symmetrically structured matrix has $k$ nonzero simple eigenvalues.}
\begin{lem}\label{lem4}
	Given an $n\times n$ symmetrically structured matrix $\bar{A}$, if $\emph{t--rank}(\bar{A})=k$, then there exists a proper variety $V_1\subset \mathbb{R}^{n_{\bar{A}}}$, such that for any numerical realization $\tilde{A}$, where the numerical values assigned to free parameters of $\bar{A}$ are encoded in the vector $\mathbf{p}_{\tilde{A}}\in \mathbb{R}^{n_{\bar{A}}}\setminus V_1$, $\tilde{A}$ has $k$ nonzero simple eigenvalues.
\end{lem}
\begin{remark}
		The challenge in the proof of Lemma~\ref{lem4} is to construct a finite number of nonzero polynomials, i.e., the polynomials of where not every coefficient is zero, such that the numerical values assigned to free parameters of $\bar{A}$ in a numerical realization $\tilde{A}$, where $\tilde{A}$ does not have $k$ nonzero simple eigenvalues, are the zeros of those polynomials. Since the set of zeros of a nonzero polynomial has Lebesgue measure zero \cite{federer2014geometric}, it follows that for any numerical realization $\tilde{A}$, $\tilde{A}$ has almost surely $k$ nonzero simple eigenvalues.
\end{remark}
\begin{remark}
	{Lemma~\ref{lem4} generally is not true for a structured matrix. For example, consider $\bar{M}=\left[\begin{smallmatrix}
	0,\star\\
	0,0
	\end{smallmatrix}\right]$, $\emph{t--rank}(\bar{M})=1$, but for any numerical realization, 
	$\tilde{M}$ has no nonzero mode.}
\end{remark}


{As shown in \cite{lin1974structural,hosoe1979irreducibility}, irreducibility is a necessary condition for structural controllability. 
We can expect that irreducibility also plays a similar role in symmetrically structured systems. Moreover, we show below that irreducibility ensures that all nonzero simple modes of $\tilde{A}$ are controllable, generically.}



\begin{lem}\label{theo2}
Given a structural pair $(\bar{A},\bar{B})$, where $\bar{A}$ is symmetrically structured and $\textrm{t--rank}(\bar{A})=k$, if $(\bar{A},\bar{B})$ is irreducible, then there exists a proper variety $V\subset \mathbb{R}^{n_{\bar{A}}+n_{\bar{B}}}$, such that for any numerical realization $(\tilde{A},\tilde{B})$ with $[\mathbf{p}_{\tilde{A}},\mathbf{p}_{\tilde{B}}]\in \mathbb{R}^{n_{\bar{A}}+n_{\bar{B}}}\setminus V$, $\tilde{A}$ has $k$ nonzero, simple and controllable modes.
\end{lem}

\subsection{Structural Controllability}
We have shown that irreducibility guarantees that generically all non-zero simple modes of $(\tilde{A},\tilde{B})$ are controllable. 
In this subsection, Theorem~\ref{theo3} proposes conditions guaranteeing that generically both the nonzero and zero modes of $(\tilde{A},\tilde{B})$ are controllable, therefore establishes a graph-theoretic necessary and sufficient condition for structural controllability in symmetrically structured system.


\begin{thm}\label{theo3}
Let $(\bar{A},\bar{B})$ be a structural pair, with $\bar{A}$ being a symmetrically structured matrix, and let $\mathcal{X}$ be the set of state vertices in $\mathcal{D}(\bar{A},\bar{B})$. The structural pair $(\bar{A}, \bar{B})$ is structurally controllable, if and only if, the following conditions hold simultaneously in $\mathcal{D}(\bar{A},\bar{B}):$
\begin{enumerate}
\item all the state vertices are input-reachable;
\item $\left|\mathcal{N(S)}\right|\ge \left|\mathcal{S}\right|$, $\forall \mathcal{S\subseteq X}.$
\end{enumerate}
\end{thm}

Notice that Conditions \emph{1)} and \emph{2)} in Theorem~\ref{theo3} admits a similar form as the conditions for structural controllability (see, for example~\cite{lin1974structural}). Subsequently, if a structural pair with symmetric parameter dependencies is structurally controllable, then the structural pair with the same structural pattern without symmetric parameter dependencies, will also be structurally controllable. 
However, the converse cannot be trivially derived due to symmetric parameter dependencies in Assumption~\ref{assumption_sym}.

\subsection{Structural Target Controllability}
We now extend the solution approach in Theorem~\ref{theo3} to establish graph-theoretic necessary and sufficient conditions for structural target controllability of the given structural pair $(\bar{A},\bar{B})$ and target set $\mathcal{T}$.


\begin{thm}\label{theo4}
Consider a structural pair $(\bar{A},\bar{B})$, with $\bar{A}$ being symmetrically structured, and a target set $\mathcal{T}\subseteq[n]$. Let $\mathcal{X_T}$ be the set of state vertices corresponding to $\mathcal{T}$ in $\mathcal{D}(\bar{A},\bar{B}).$ The structural pair $(\bar{A}, \bar{B})$ is structurally target controllable with respect to $\mathcal{T}$, if and only if, the following conditions hold simultaneously in $\mathcal{D}(\bar{A},\bar{B}):$
	\begin{enumerate}
			\item all the states vertices in $\mathcal{X_T}$ are input-reachable;
			\item $\left|\mathcal{N(S)}\right|\ge\left|\mathcal{S}\right|$, $\forall \mathcal{S}\subseteq \mathcal{X_T}$.
		\end{enumerate}
\end{thm}

\begin{remark}
Condition 2) in Theorem~\ref{theo4} can be verified using local topological information in the network. In particular, this condition is satisfied if there exists a matching in the bipartite graph $\mathcal{B}(\mathcal{S}_1,\mathcal{S}_2,\mathcal{E}_{\mathcal{S}_1,\mathcal{S}_2})$ associated with $\mathcal{D}(\bar{A},\bar{B})$, where $\mathcal{S}_1=\mathcal{X\cup U}$ and $\mathcal{S}_2=\mathcal{X_T}$, such that all vertices in $\mathcal{S}_2$ are right-matched. The existence of such a matching can be verified in $\mathcal{O}(\sqrt{|\mathcal{S}_1\cup \mathcal{S}_2|}|\mathcal{E}_{\mathcal{S}_1, \mathcal{S}_2}|)$ time \cite[\S 23.6]{cormen2009introduction}.
\end{remark}
{Through the proof of Theorem~\ref{theo4}, we notice that the characterization of structural target controllability relies on the assumption that the state matrix is symmetric. }More specifically, since the state matrix is symmetric, the eigenvectors of the state matrix form a complete basis of the state space, which allows us to generalize the PBH test in the context of target controllability problems. On the contrary, when the system is characterized by a directed network, the state matrix $A$ is, in general, non-diagonalizable, which prevents us from generalizing PBH test to characterize the target controllability problems - see \cite[Example 3]{murota1990note} for a reference. 

In addition, the proof of Theorem~\ref{theo4} suggests that, even for the case where $\bar{A}$ is not symmetrically structured, the violation of either Conditions~\emph{1)} or \emph{2)} results in that $(\bar{A},\bar{B})$ is not structurally target controllable. Therefore, in general, when the structured matrix $\bar{A}\in\{0,\star\}^{n\times n}$ is not symmetrically structured, the Conditions \emph{1)} and \emph{2)} in Theorem~\ref{theo4} are necessary but not sufficient conditions for the structural target controllability of the pair $(\bar{A},\bar{B})$.

\section{Illustrative Examples}\label{sec:example}
In this section, we provide an example to illustrate our necessary and sufficient conditions 
in Theorem~\ref{theo3} and Theorem~\ref{theo4}. We consider a symmetrically structured system with $10$ states and $2$ inputs modeled by an undirected network with unknown link weights. The structural representations of its state and input matrix are denoted by $\bar{A} \in \{0,\star\}^{10\times 10}$ and $\bar{B}\in \{0,\star\}^{10 \times 2},$ as follows. 
\begin{equation*}\label{eq:SystemExample}\footnotesize
\bar{A}=\left[\begin{smallmatrix}
0&a_{12}&0&a_{14}&a_{15}&0&0&0&0&0\\
a_{12}&0&a_{23}&0&0&0&0&0&0&0\\
0&a_{23}&0&a_{34}&0&0&0&0&0&0\\
a_{14}&0&a_{34}&0&0&0&0&0&0&0\\
a_{15}&0&0&0&0&0&a_{57}&0&0&0\\
0&0&0&0&0&a_{66}&a_{67}&0&0&0\\
0&0&0&0&a_{57}&a_{67}&0&0&a_{79}&0\\
0&0&0&0&0&0&0&0&a_{89}&0\\
0&0&0&0&0&0&a_{79}&a_{89}&0&a_{910}\\
0&0&0&0&0&0&0&0&a_{910}&0
\end{smallmatrix}\right],\bar{B}=\left[\begin{smallmatrix}
0,b_{12}\\b_{21},0\\0,0\\0,0\\0,b_{52}\\0,0\\0,0\\0,0\\0,0\\0,0
\end{smallmatrix}\right].
\end{equation*}
In addition, we let the target set be $\mathcal{T}=\{2,6,8\}.$ Subsequently, $C_{\mathcal{T}}$, defined according to~\eqref{CT}, equals to 
		

\begin{equation*}\footnotesize
	C_{\mathcal{T}}=\left[\begin{smallmatrix}
	0&1&0&0&0&0&0&0&0&0\\
	0&0&0&0&0&1&0&0&0&0\\
	0&0&0&0&0&0&0&1&0&0
	\end{smallmatrix}\right].
\end{equation*}
 We also associate the structural pair $(\bar{A},\bar{B})$ with the digraph $\mathcal{D}(\bar{A},\bar{B})=(\mathcal{X}\cup\mathcal{U}, \mathcal{E}_{\mathcal{X},\mathcal{X}}\cup \mathcal{E}_{\mathcal{U},\mathcal{X}})$ as depicted in Figure~\ref{fig2},
		\begin{figure}[t]
			\centering
			\includegraphics[width=0.2\textwidth]{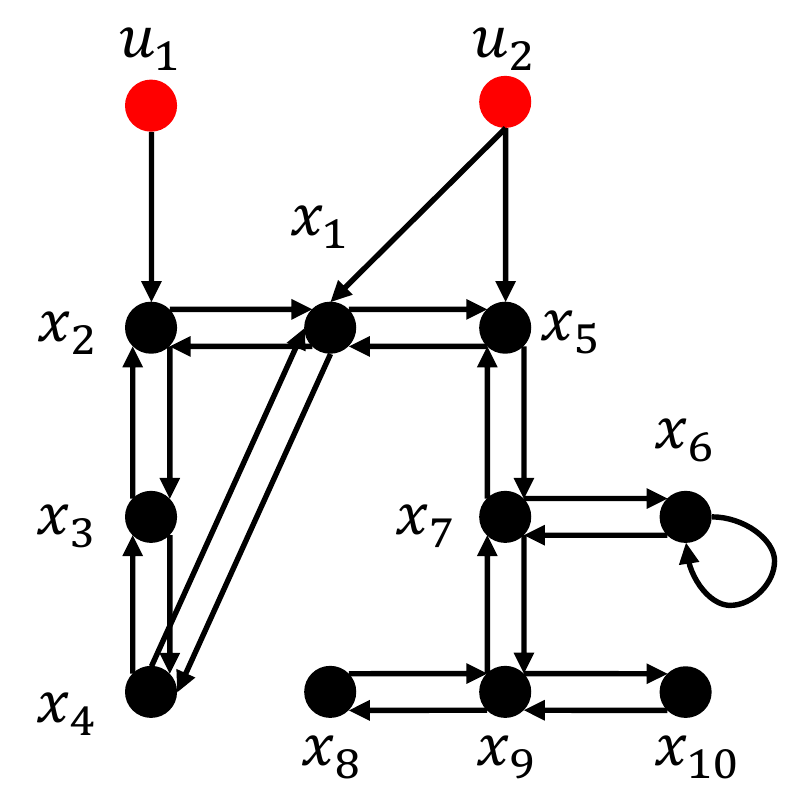}
			\caption{the digraph representation of the structural pair $(\bar A, \bar B)$, where the red and black vertices represent input and state vertices, respectively. The black arrows represent edges in $\mathcal{D}(\bar{A},\bar{B})$.}
			\label{fig2}
		\end{figure}
	 
		\noindent where $\mathcal{X}=\{x_1,\dots,x_{10}\}$, $\mathcal{U}=\{u_1,u_2\}$ and $\mathcal{X_T}=\{x_2,x_6,x_8\}$. 
		
		Notice that by letting $\mathcal{S}=\{x_8,x_{10}\}$, we have $\mathcal{N(S)}=\{x_9\}$. As a result, according to Theorem~\ref{theo3}, $\exists \mathcal{S}\subseteq\mathcal{X}$, $\left|\mathcal{N(S)}\right|<\left|\mathcal{S}\right|$ implies that the system is not structurally controllable. 
		However, since all the vertices in $\mathcal{X_T}$ are input-reachable, and $|\mathcal{N(S)}|\ge|\mathcal{S}|,\ \forall\mathcal{S\subseteq X_T}$, by Theorem~\ref{theo4}, $(\bar{A},\bar{B})$ is structurally target controllable with respect to $\mathcal{T}$. This example also shows that, if the input-reachability of the vertices in $\mathcal{X_T}$ is guaranteed, then the structural target controllability in undirected networks can be verified by only local topological information.

\section{Conclusions}\label{sec:con}
{In this paper, we study the problem of characterizing structural target controllability in undirected networks with unknown link weights. 
We achieved this goal by first characterizing the generic properties of symmetrically structured matrices. We then proposed a necessary and sufficient condition for structural controllability of undirected networks with multiple control inputs. 
Finally, we provided a graph-theoretic necessary and sufficient condition for structural target controllability of undirected networks. In the future, we will use the conditions to implement algorithms for designing minimum number of input actuations to ensure the structural target controllability of undirected networks.

\section*{Appendix}
\subsection{Proof of Lemma~\ref{symmetric} and related results}
Before we proceed to the proof of Lemma~\ref{symmetric}, we introduce Proposition~\ref{termrank}, which lays the foundation for the proof of Lemma~\ref{symmetric}.
\begin{prop}[{\cite[\S1.2]{murota2012systems}}]\label{termrank}
	Given a (symmetrically) structured matrix $\bar{M}\in\{0,\star\}^{n\times m}$ and $\mathcal{B}(\mathcal{S}_1, \mathcal{S}_2,\mathcal{E}_{\mathcal{S}_1,\mathcal{S}_2})$, where $\mathcal{S}_1=\{v_1,\dots, v_m\}$, $\mathcal{S}_2=\{v'_1,\dots,v'_n\}$ and $\mathcal{E}_{\mathcal{S}_1,\mathcal{S}_2}=\{\{v_i,v'_j\}\colon [\bar{M}]_{ji}\ne0,v_i\in\mathcal{S}_1,v'_j\in\mathcal{S}_2\}$, then $\emph{t--rank}(\bar{M})=n$ if and only if $\left|\mathcal{N_B(S)}\right|\ge\left|\mathcal{S}\right|$ for all $\mathcal{S}\subseteq\mathcal{S}_2$.
\end{prop}
\begin{proof}[Proof of Lemma~\ref{symmetric}]
	First, we show the sufficiency of the theorem. Notice that the generic-rank of $C_{\mathcal{T}}\bar{A}$ equals $k$, if and only if, there exists a $k$-by-$k$ non-zero minor in $C_{\mathcal{T}}\bar{A};$ hence, it suffices to find that minor. Since $\left|\mathcal{N(S)}\right|\ge \left|\mathcal{S}\right|$, $\forall \mathcal{S}\subseteq \mathcal{X_T}$, there exists $k$ entries that lie on distinct rows and distinct columns of $C_{\mathcal{T}}\bar{A}$ according to Proposition~\ref{termrank}. As a result, we can select rows indexed by $\mathcal{T}=\{{i_1},\dots,{i_k}\}$ and columns indexed ${j_1},\cdots,{j_k}$ in $\bar{A}$ such that $\{[\bar{A}]_{i_\ell j_\ell}\}_{\ell=1}^k
	$ lies on distinct rows and distinct columns. Next, we consider the following two cases.
	
	On one hand, if $\{j_1,\dots,j_k \}=\{i_1,\dots,i_k\},$ then $M = C_{\mathcal{T}}\bar{A}C_{\mathcal{T}}^\top$ is a square submatrix of $\bar{A}.$ We consider a particular numerical realization $\tilde{A}$ of $\bar{A}$, as follows. Let $[\tilde{A}]_{ij} \neq 0$ for all $(i,j) \notin \{(i_\ell, j_\ell) : \ell \in [k]\},$  $[\tilde{A}]_{ij} = [\tilde{A}]_{ji}$, and $[\tilde{A}]_{ij}=0$ otherwise. Subsequently, by computing the determinant
	, $\det (C_{\mathcal{T}}\tilde{A}C_{\mathcal{T}}^\top) = \textrm{sgn}(\sigma_1) \Pi_{\ell = 1}^k [\tilde{A}]_{i_\ell j_\ell} + \textrm{sgn}(\sigma_2) \Pi_{\ell = 1}^k [\tilde{A}]_{j_\ell i_\ell},$ where $\textrm{sgn}(\sigma_1)$ and $\textrm{sgn}(\sigma_2)$ are the signatures of the permutations $\sigma_1 = \{(i_\ell, j_\ell) : \ell \in [k]\},$ and $\sigma_2 = \{(j_\ell, i_\ell) : \ell \in [k]\},$ respectively. Notice that if $\textrm{sgn}(\sigma_1) = \textrm{sgn}(\sigma_2),$ then it follows that $\det (C_{\mathcal{T}}\tilde{A}C_{\mathcal{T}}^\top) \neq 0.$ Furthermore, if $\{ \tilde{A} : \det (C_{\mathcal{T}}\tilde{A}C_{\mathcal{T}}^\top) = 0\}$ is a proper variety, we have that $M$ admits an $k$-by-$k$ non-zero minor generically. Thus, the generic-rank of $C_{\mathcal{T}}\bar{A}$ equals to $k.$
	
	On the other hand, when $\{j_1,\dots,j_k\}\ne\{i_1,\cdots,i_k\}$, it sufficies to show there exists a numerical realization $\tilde{A}$ such that $\det ([\tilde{A}]_{i_1,\cdots,i_k}^{j_1,\cdots,j_k})\ne0$. We consider a numerical realization $\tilde{A}$ by assigning distinct real values to $\star$-entries corresponding to $\{[\bar{A}]_{i_\ell j_\ell}\}_{\ell=1}^k$ while keeping $[\tilde{A}]_{ij}=[\tilde{A}]_{ji}$, and assigning 0 otherwise. 
	Without loss of generality, we can permute $\{\ell\}_{\ell=1}^k$ such that for each $[\bar{A}]_{i_{\ell_r}j_{\ell_r}}\in\{[\bar{A}]_{i_{\ell_r}j_{\ell_r}}\}_{r=1}^p$, $[\bar{A}]_{j_{\ell_r}i_{\ell_r}}$ is not in matrix $[\bar{A}]_{i_1,\cdots,i_k}^{j_1,\cdots,j_k}$, and for each $[\bar{A}]_{i_{\ell_r}j_{\ell_r}}\in\{[\bar{A}]_{i_{\ell_r}j_{\ell_r}}\}_{r=p+1}^k$, $[\bar{A}]_{j_{\ell_r}i_{\ell_r}}$ is in matrix $[\bar{A}]_{i_1,\cdots,i_k}^{j_1,\cdots,j_k}$. We declaim that there is only one nonzero entry in either the $i_{\ell_r}$th row or $j_{\ell_r}$th column, $\forall r\in[p]$, otherwise it contradicts that $\{[\bar{A}]_{i_\ell j_\ell}\}_{\ell=1}^k$ are in distinct rows and distinct columns of $[\bar{A}]$. Thus, we compute $\det ([\tilde{A}]_{i_1,\cdots,i_k}^{j_1,\cdots,j_k})$,
	\begin{equation}\small\label{product}
	\det ([\tilde{A}]_{i_1,\cdots,i_k}^{j_1,\cdots,j_k})=(\prod_{r=1}^{p}[\tilde{A}]_{i_{\ell_r}j_{\ell_r}})\cdot \det([\tilde{A}]_{i_{\ell_{p+1}},\cdots,i_{\ell_{k}}}^{j_{\ell_{p+1}},\cdots,j_{\ell_k}})\ne 0,
	\end{equation}
	where $\det([\tilde{A}]_{i_{\ell_{p+1}},\cdots,i_{\ell_{k}}}^{j_{\ell_{p+1}},\cdots,j_{\ell_k}})\ne0$ is true because of the reasoning in the first case $\{i_1,\cdots,i_k\}=\{j_1,\cdots,j_k\}$. Thus, there exists numerical realization such that $\det([\tilde{A}]_{i_1,\cdots,i_k}^{j_1,\cdots,j_k})\ne0$.

	Next, we show the necessity of the theorem by contrapositive. We assume that there exists $\mathcal{S}\subseteq\mathcal{X_T}$, such that $\left|\mathcal{N(S)}\right|<\left|\mathcal{S}\right|.$ Then, by Proposition~\ref{termrank}, there does not exist $k$ entries that lie on the distinct rows and distinct columns of $C_{\mathcal{T}}\bar{A},$ which implies $\textrm{g--rank}(C_{\mathcal{T}}\bar{A})<k$.
\end{proof}


\subsection{Proof of Corollary~\ref{coro}}\label{Acoro}
	\begin{proof}
	Suppose $\left|\mathcal{N	(S)}\right|\ge\left|\mathcal{S}\right|$, $\forall \mathcal{S}\subseteq\mathcal{X_T}$, then, by Proposition~\ref{termrank}, there exist $k$ entries, $\{[\bar{A},\bar{B}]_{i_\ell j_\ell}\}_{\ell=1}^k$, such that they are all $\star$-entries which lie on distinct rows and distinct columns of $[\bar{A},\bar{B}]$. 
	Among those $k$ entries, suppose $\{[\bar{A},\bar{B}]_{i_\ell j_\ell}\}_{\ell=1}^q$ are in columns of $\bar{A}$, and $\{[\bar{A},\bar{B}]_{i_\ell j_\ell}\}_{\ell={q+1}}^k$ are in the columns of $\bar{B}$. By Lemma~\ref{symmetric}, there exists a numerical realization $\tilde{A}$, such that $\det([\tilde{A},\tilde{B}]_{i_{1},\dots,i_q}^{j_{1},\dots,j_q})\ne0$. Since $\bar{B}$ is a structured matrix, there exists a numerical realization $\tilde{B}$ such that $\det([\tilde{A},\tilde{B}]_{i_{q+1},\dots,i_k}^{j_{q+1},\dots,j_k})\ne0$ Hence, there exists a numerical realization $[\tilde{A},\tilde{B}]$ with
	\begin{equation*}\label{minor}\small
	\begin{aligned}
	\det([\tilde{A},\tilde{B}]_{i_1,\dots,i_k}^{j_1,\dots,j_k})&=\det([\tilde{A},\tilde{B}]_{i_1,\dots,i_q}^{j_1,\dots,j_q})\det([\tilde{A},\tilde{B}]_{i_{q+1},\dots,i_k}^{j_{q+1},\dots,j_k})\\&\ne0,
	\end{aligned}
	\end{equation*}
	
\noindent which implies that $\textrm{g--rank}(C_{\mathcal{T}}\left[\bar{A},\bar{B}\right])=k$.
\end{proof}


\subsection{Proof of Lemma~\ref{lem4}}
We introduce Proposition~\ref{prop}, Proposition~\ref{hoff} and Lemma~\ref{lem3} to support the proof of Lemma~\ref{lem4}.
\begin{prop}[{\cite[\S 2.1]{barnett1971matrices}}]\label{prop}
	Let $\varphi_1(s)$ and $\varphi_2(s)$ be polynomials in $s$ with $\varphi_1(s)=\sum_{i=0}^{n_1}a_is^{{n_1}-i}$, and $\varphi_2(s)=\sum_{i=0}^{n_2}b_is^{n_2-i}$, respectively. Let $R(\varphi_1,\varphi_2)$ be defined as
	\begin{equation}\label{sylv}\small
	R(\varphi_1,\varphi_2)=\det\left(\left[\begin{smallmatrix}
	a_{n_1} & a_{n_1-1} & \cdots & a_0 & 0 & \cdots & 0\\
	0 & a_{n_1} & \cdots & a_{1} & a_{0} & \cdots & 0\\
	\vdots & \vdots & \ddots & \vdots & \vdots & \ddots & \vdots\\
	0 & 0 &\cdots & a_{n_1} & a_{n_1-1} & \cdots & a_{0}\\\hline
	0 & 0 & \cdots &		  &         & \cdots & b_0\\
	\vdots & \vdots & \iddots & \vdots & \vdots & \iddots & \vdots\\
	0 & b_{n_2} & \cdots & b_1& b_0& \cdots & 0\\
	b_{n_2} & b_{n_2-1} & \cdots & b_0& 0& \cdots & 0  
	\end{smallmatrix}\right]\right).
	\end{equation}
	If $a_{n_1}\ne0$ and $b_{n_2}\ne 0$, then $\varphi_1(s)$ and $\varphi_2(s)$ have a nontrivial common factor if and only if the $R(\varphi_1,\varphi_2)=0.$
\end{prop}

\begin{prop}[Hoffman-Wielandt Theorem {\cite[\S6.3]{horn1990matrix}}]\label{hoff}
	Given $n\times n$ symmetric matrices $A$ and $E$, let $\lambda_1,\dots,\lambda_n$ be the eigenvalues of $A$, and $\hat{\lambda}_1,\dots,\hat{\lambda}_n$ be the eigenvalues of $A+E$. There is a permutation $\sigma(\cdot)$ of the integers $\{1,\dots,n\}$ such that
	\begin{equation}\label{hoffequation}\small
	\sum_{i=1}^n(\hat{\lambda}_{\sigma(i)}-\lambda_i)^2\le\norm{E}_F^2,
	\end{equation}
	
	\noindent where $\norm{E}_F = \sqrt{\textrm{tr}(E E^\top)}.$
\end{prop}
\begin{lem}\label{lem3}
	Let $\bar{A}$ be an $n\times n$ symmetrically structured matrix, and let $\mathcal{D}(\bar{A})=\{\mathcal{X},\mathcal{E_{X,X}}\}$ be the digraph associated with $\bar{A}$. Assume \emph{t--rank}$(\bar{A})=k$, and denote $\{ [\bar{A}]_{i_\ell j_\ell} \}_{\ell = 1}^k$ as the $k$ entries that lie on distinct rows and distinct columns. We define $\mathcal{S}=\{x_{i_1},\dots,x_{i_k}\}\subseteq \mathcal{X}$. Then, $\mathcal{D}_{\mathcal{S}}$ can be covered by disjoint cycles.
\end{lem}
	\begin{proof}[Proof of Lemma~\ref{lem3}]
	We approach the proof by contradiction. Suppose $\mathcal{D}_{\mathcal{S}}$ cannot be covered by disjoint cycles, then at least one vertex $x_{i}\in\mathcal{S}$ can only be covered by cycles intersecting with other cycles in $\mathcal{D}_{\mathcal{S}}$, 
	which implies that there does not exist $k$ edges in which no two edges share the same 'tail' or 'head' vertex in $\mathcal{D}(\bar{A})$, i.e., there does not exist $k$ entries that lie on distinct rows and distinct columns of $\bar{A}$, which, by Proposition~\ref{termrank}, contradicts $\textrm{t--rank}(\bar{A})=k$.
\end{proof}

\begin{proof}[Proof of Lemma~\ref{lem4}] 
	We expand the characteristic polynomial of a matrix~$\tilde{A}$ as
	\begin{equation}\label{l2e1}\small
	\det(sI-\tilde{A})=s^n+a_{n-1}s^{n-1}\dots+a_{n-k}s^{n-k}+\dots + a_0.
	\end{equation}
	Besides, we have
	\begin{equation}\label{l2e2}\small
	a_q=(-1)^{n-q}\sum_{1\le k_1<\dots<k_{n-q} \le 	n}\det([\tilde{A}]_{k_1,\dots,k_{n-q}}^{k_1,\dots,k_{n-q}}),
	\end{equation}
	where $q=0,1,\dots,n-1$. Since $\textrm{t--rank}(\bar{A})=k$, there exists a numerical realization $\tilde{A}$ and a set of indexes, $\{i_1,\dots,i_k\}\subseteq[n]$, such that $\det([\tilde{A}]_{i_1,\dots,i_k}^{i_1,\dots,i_k})\ne0.$ Furthermore, $V_{0}:=\{\mathbf{p}_{\tilde{A}}\in\mathbb{R}^{n_{\bar{A}}}:a_{n-k}=0\}$ is a proper variety. Since the maximum order of principle minor is at most the term rank of a matrix, we have $a_{n-k-1}=\dots =a_0=0$. Thus, to characterize nonzero eigenvalues, we define the polynomial $\varphi_{\tilde{A}}(s)$ as
	\begin{equation}\label{phi}\small
		\varphi_{\tilde{A}}(s)=s^k+a_{n-1}s^{k-1}+\dots+a_{n-k}.
	\end{equation}
	
	In the rest of the proof, we show that there exists a numerical realization $\mathbf{p}_{\tilde{A}}\in V_0^c$ such that $\tilde{A}$ has $k$ non-zero simple eigenvalues. Since $\textrm{t--rank}(\bar{A})=k$, we define the set $\mathcal{S}$ as in Lemma~\ref{lem3}. By Lemma~\ref{lem3}, there exist disjoint cycles $\mathcal{C}_1,\dots,\mathcal{C}_l$ covering $\mathcal{D}_{\mathcal{S}}$. Let us denote by $\mathcal{C}_i$ the $i$-th cycle in $\{\mathcal{C}_1,\dots,\mathcal{C}_l\}$. Moreover, without loss of generality, we let the length of cycle $\mathcal{C}_i$ be either $\left|\mathcal{C}_i\right|=2q$, or $\left|\mathcal{C}_i\right|=2q+1$, for some $q\in\mathbb{N}$. 
	{Note that by definition, there is a one-to-one correspondence between the edge in $\mathcal{D}(\bar{A})$ and the $\star$-entry in $\bar{A}$. From this observation, we denote by $\bar{A}_i\in\{0,\star\}^{|\mathcal{C}_i|\times |\mathcal{C}_i|}$ the square submatrix formed by collecting rows and columns corresponding to the indexes of vertices in $\mathcal{V}_{\mathcal{C}_i}$ of the cycle $\mathcal{C}_i$. }
	We let all the $\star$-entries of $\bar{A}$ be zero, except for $\star$-entries corresponding to edges in $\{\mathcal{E}_{\mathcal{C}_i}\}_{i=1}^l$. Hence, there exists a permutation matrix $P$ and numerical realization $\tilde{A}$, such that $P\tilde{A}P^{-1}$ is a block diagonal matrix,
	\begin{equation}\label{decomposition}\small
	P\tilde{A}P^{-1}=\begin{bmatrix}
	\tilde{A}_{1}&\mathbf{0}&\cdots&\mathbf{0}&\mathbf{0}\\
	\mathbf{0}&\tilde{A}_2&\cdots&\mathbf{0}&\mathbf{0}\\
	\vdots & \vdots & \ddots&\vdots&\vdots\\
	\mathbf{0} & \mathbf{0} &\cdots&\tilde{A}_l&\mathbf{0}\\
	\mathbf{0} & \mathbf{0} & \cdots &\mathbf{0}&\mathbf{0}
	\end{bmatrix}.
	\end{equation}
	
	If $\left|\mathcal{C}_i\right|=2q$, without loss of generality, we could assume $\mathcal{C}_i=(x_{i_1},x_{j_1},x_{i_2},x_{j_2},\dots,x_{i_q},x_{j_q},x_{i_1})$. Since $\mathcal{D}_{\mathcal{V}_{\mathcal{C}_i}}$ is a subgraph of the digraph $\mathcal{D}(\bar{A})$ associated with the symmetrically structured matrix $\bar{A}$, there exist $q$ disjoint cycles of length-2 covering $\mathcal{D}_{\mathcal{V}_{\mathcal{C}_i}}$, i.e., cycles $(x_{i_1},x_{j_1},x_{i_1}),(x_{i_2},{x_{j_2}},{x_{i_2}}),\dots,({x_{i_q}},{x_{j_{q}}},{x_{i_q}})$. 
	We assign distinct nonzero weights to $\star$-entries of $\bar{A}_i$ that correspond to edges in the $q$ cycles of length-2, and assign zero weights to other $\star$-entries in $\bar{A}_i$. As a result, we have
	
	\begin{equation*}\tiny
	\tilde{A}_i=\left[\begin{array}{ccccccccccccccc}
	\tikzmark{left1} 0&a_{i_1j_1}&\cdots&0&0\\
	a_{i_1j_1}&0\tikzmark{right1}&\cdots&0&0\\
	\vdots&\vdots&\ddots&\vdots&\vdots\\
	0&0&\cdots&\tikzmark{left2}0&a_{i_qj_q}\\
	0&0&\cdots&a_{i_qj_q}&0\tikzmark{right2}
	\end{array}\right],
	\tikzdrawbox{1}{black}
	\tikzdrawbox{2}{black}
	\end{equation*}

	\noindent where $a_{i_1j_1},\dots,a_{i_qj_q}$ are $q$ nonzero distinct weights. Thus, $\tilde{A}_i$ has $2q$ simple nonzero eigenvalues.
	
	If $\left|\mathcal{C}_i\right|=1$, then the eigenvalue of $\tilde{A}_i\in\mathbb{R}^{1\times 1}$ can be placed to any value. If $\left|\mathcal{C}_i\right|=2q+1$ and $q>0$, then there are $2q$ vertices in $\mathcal{C}_i$ that can be covered by $q$ cycles of length-$2$, and one vertex that cannot be covered by any length-$2$ cycle in a vertex-disjoint way in $\mathcal{D}_{\mathcal{V}_{\mathcal{C}_i}}$. Assign distinct nonzero weights to $\star$-entries corresponding to the $q$ cycles of length-2, and zero to other $\star$-entries in $\bar{A}_i$. As a result, the constructed numerical realization, $\tilde{A}_i$, has $2q$ nonzero simple eigenvalues and one zero eigenvalue. Denote by $\lambda_j(\tilde{A}_i)$ the $j$th eigenvalue of $\tilde{A}_i$, $j\in\{1,\dots,\left|\mathcal{C}_i\right|\}$.
	
	By Proposition~\ref{hoff}, given a sufficiently small $\epsilon>0,$ $\exists\delta>0$ and permutation $\sigma(\cdot)$ of integers $\{1,\dots,|\mathcal{C}_i|\}$, such that for two numerical realizations of $\bar{A}_i$: $\tilde{A}_i$ and $\tilde{A}_{ip}$, if $||\tilde{A}_{ip}-\tilde{A}_i||_F<\delta,$ then $\max\{|\lambda_{\sigma(j)}(\tilde{A}_{ip})-\lambda_{j}(\tilde{A}_i)|\}<\epsilon$. Perturb $\star$-entries of $\tilde{A}_i$ corresponding to edges in $\mathcal{E}_{{\mathcal{C}}_i}$, 
	such that $\tilde{A}_{ip}$, which is derived by this perturbation of $\tilde{A}_i$, satisfies $||\tilde{A}_{ip}-\tilde{A}_i||_F<\delta$. Moreover, since $\textrm{t--rank}(\bar{A}_i)=2q+1$, by Lemma~\ref{symmetric}, $\textrm{g--rank}(\bar{A}_i)=2q+1$. 
	{The above analysis shows that we can perturb $\tilde{A}_i$, such that $\textrm{rank}(\tilde{A}_{ip})=2q+1$, and}
	\begin{equation*}
	\small
	\begin{aligned}
	&\min_{
		j\ne r,j,r\in\{1,\dots,|\mathcal{C}_i|\}	}	|\lambda_{j}(\tilde{A}_{ip})-\lambda_{r}(\tilde{A}_{ip})|>\\
	&\indent\min_{j\ne r,j,r\in\{1,\dots,|\mathcal{C}_i|\}}|\lambda_j(\tilde{A}_i)-\lambda_r(\tilde{A}_i)|-2\epsilon.
	\end{aligned}
	\end{equation*}
	It implies that there exists $\tilde{A}_{ip}$ which has $2q+1$ nonzero simple eigenvalues. Notice that $\tilde{A}_{ip}$ is also a numerical realization of $\bar{A}_i$. Hence, for either $\left|\mathcal{C}_i\right|=2q$, or $\left|\mathcal{C}_i\right|=2q+1$, there exists a numerical realization $\tilde{A}_i$ such that $\tilde{A}_i$ has $\left|\mathcal{C}_i\right|$ nonzero simple eigenvalues. Also, there exists $\tilde{A}$ 
	that has $\sum_{i=1}^l\left|\mathcal{C}_i\right|=k$ nonzero simple eigenvalues.
	
	Denote by $\varphi'_{\tilde{A}}$ the derivative of $\varphi_{\tilde{A}}$ with respect to $\lambda$. If~$\mathbf{p}_{\tilde{A}}\in V^c_0$, and $\tilde{A}$ has repeated nonzero modes, then $\varphi_{\tilde{A}}$ and $\varphi'_{\tilde{A}}$ have a common nontrivial zero (i.e., by Proposition~\ref{prop}, $R(\varphi_{\tilde{A}},\varphi'_{\tilde{A}})=0$). Define $V_1=\{\mathbf{p}_{\tilde{A}}\in\mathbb{R}^{n_{\bar{A}}}\colon a_{n-k}=0\ \textrm{or}\ R(\varphi_{\tilde{A}},\varphi'_{\tilde{A}})=0\}$, where $a_{n-k}=0$ and $R(\varphi_{\tilde{A}},\varphi'_{\tilde{A}})=0$ are both polynomials of $\star$-entries of $\bar{A}$. Since we have shown that there exists $\tilde{A}$ which has $k$ nonzero simple eigenvalues, i.e., $\exists \mathbf{p}_{\tilde{A}}\in\mathbb{R}^{n_{\bar{A}}}$ such that $a_{n-k}\ne0$ and $R(\varphi_{\tilde{A}},\varphi'_{\tilde{A}})\ne 0$, we conclude that $V_1$ is proper.
\end{proof}

\begin{remark}
	To characterize the generic rank of $[\bar{A},\bar{B}]$, which is crucial in the derivation of Lemma~\ref{theo2}, we should consider the proper variety in parameter space $\mathbb{R}^{n_{\bar{A}}+n_{\bar{B}}}$. Since each $\star$-entry of $\bar{A}$ is independent of those in $\bar{B}$, $V_1$ is also a proper variety in $\mathbb{R}^{n_{\bar{A}}+n_{\bar{B}}}$. Let us redefine $V_1$ as
	\begin{equation}\label{V_1}\small
	V_1=\{[\mathbf{p}_{\tilde{A}},\mathbf{p}_{\tilde{B}}]\in\mathbb{R}^{n_{\bar{A}}+n_{\bar{B}}}:a_{n-k}=0\ \emph{or}\ R(\varphi_{\tilde{A}},\varphi'_{\tilde{A}})=0 \}.
	\end{equation}
\end{remark}

\subsection{Proof of Lemma~\ref{theo2}}
We first introduce Lemma~\ref{lem5} in support of proving Lemma~\ref{theo2}.
\begin{lem}\label{lem5}
	{Consider an irreducible structural pair $(\bar{A},\bar{B})$, where $\bar{A}\in\{0,\star\}^{n\times n}$ is a symmetrically structured matrix with $\emph{t--rank}(\bar{A})=k$. Let $V_1\subset\mathbb{R}^{n_{\bar{A}}+n_{\bar{B}}}$ be defined as in \eqref{V_1}. There exists a proper variety $V_2\subset \mathbb{R}^{n_{\bar{A}}+n_{\bar{B}}}$ such that if $[\mathbf{p}_{\tilde{A}},\mathbf{p}_{\tilde{B}}]\in V^c_1$, then there exists a non-zero uncontrollable mode of $\tilde{A}$ if and only if $[\mathbf{p}_{\tilde{A}},\mathbf{p}_{\tilde{B}}]\in V_2$.}
\end{lem}
\begin{proof}[Sketch of Proof of Lemma~\ref{lem5}]
	We will first prove that $V_2$ exists. Suppose $[\mathbf{p}_{\tilde{A}},\mathbf{p}_{\tilde{B}}]\in V^c_1$, by a similar reasoning as in Lemma~\ref{lem4}, all the $k$ nonzero eigenvalues of $\tilde{A}$ are simple. Let $\lambda$ be a nonzero eigenvalue of $\tilde{A}$, and $\varphi_{\tilde{A}}(s)$ be defined as in \eqref{phi}, then we have,
	\begin{equation}\label{lambda}\small
	\varphi_{\tilde{A}}(\lambda)=\lambda^{k}+a_{n-1}\lambda^{k-1}+\dots+a_{n-k}=0.
	\end{equation}
	Let us further assume that $(\lambda,v)$ is an uncontrollable mode of $\tilde{A}$; in other words,
	\begin{equation}\small
	v^\top\tilde{A}=\lambda v^\top,\ \ 
	v^\top\tilde{B}=\mathbf{0}\label{v21}.
	\end{equation}
	Since all the nonzero eigenvalues $\lambda$ are simple, recall the fact in \cite{hosoe1979irreducibility} that the left eigenvector $v^\top$ equals (apart from a constant scalar) any of the nonzero row of the adjugate matrix $\textrm{adj}(\lambda I-\tilde{A})$. Hence,
	\begin{equation}\small
	\textrm{adj}(\lambda I-\tilde{A})\tilde{B}=\mathbf{0}_{n\times m}\label{v22}.
	\end{equation}
	Equations \eqref{lambda} and \eqref{v22} imply that the two polynomials \eqref{v23} and \eqref{v24} have a common zero $\lambda$, namely,
	\begin{small}
		\begin{align}
		\varphi_{\tilde{A}}(s)&=s^k+a_{n-1}s^{k-1}+\dots+a_{n-k}=0\label{v23},\\
		\psi_{\tilde{A},\tilde{B}}(s)&=\textrm{tr}([\textrm{adj}(sI-\tilde{A})\tilde{B}][\textrm{adj}(sI-\tilde{A})\tilde{B}]^\top)=0\label{v24}.
		\end{align}
	\end{small}

\noindent{The} variety $V_2$ is defined as follows,
	\begin{equation}\label{V_2}\small
	V_2=\{[\mathbf{p}_{\tilde{A}},\mathbf{p}_{\tilde{B}}]\in\mathbb{R}^{n_{\bar{A}}+n_{\bar{B}}}:R(\varphi_{\tilde{A}},\psi_{\tilde{A},\tilde{B}})=0 \},
	\end{equation}  
	where $R(\varphi_{\tilde{A}},\psi_{\tilde{A},\tilde{B}})=0$ is a polynomial of the $\star$-entries in $\bar{A}$ and $\bar{B}$. The properness of $V_{2}$ can be shown by contradiction by adapting the proof in \cite[Theorem 2]{hosoe1979irreducibility}.
Conversely, suppose $[\mathbf{p}_{\tilde{A}},\mathbf{p}_{\tilde{B}}]\in V_2\cap V^c_1$, by the definition of $V_1$ and $V_2$, $\varphi_{\tilde{A}}$ and $\psi_{\tilde{A},\tilde{B}}$ have a common zero $\lambda\ne0$. Since $\lambda$ is a zero of $\varphi_{\tilde{A}}$, $\lambda$ is also an eigenvalue of $\tilde{A}$, which is an uncontrollable eigenvalue.
\end{proof}
\begin{proof}[Proof of Lemma~\ref{theo2}]
Define $V=V_1\cup V_2$, where $V_1$ and $V_2$ are defined as in (\ref{V_1}) and (\ref{V_2}), respectively. We can prove $V_1$ is proper by a similar reasoning as the one in Lemma~\ref{lem4}. By Lemma~\ref{lem5}, $V_2$ is proper. Hence, $V=V_1\cup V_2$ is proper. If $[\mathbf{p}_{\tilde{A}},\mathbf{p}_{\tilde{B}}]\in V^c$, $\tilde{A}$ has $k$ nonzero simple controllable modes.
\end{proof}
\subsection{Proof of Theorem~\ref{theo3}}
We first introduce Lemma~\ref{nec}, which lays the foundation for the proof of Theorem~\ref{theo3}.
\begin{lem}\label{nec}
	{Consider a structural pair $(\bar{A},\bar{B})$, and a target set $\mathcal{T}$ with the corresponding state vertex set $\mathcal{X_T}$ in $\mathcal{D}(\bar{A},\bar{B})$. We define $C_{\mathcal{T}}$ according to (\ref{CT}). Given a numerical realization $(\tilde{A},\tilde{B})$, we define controllability matrix $Q(\tilde{A},\tilde{B})$ as in Definition~\ref{structuralcontrollability}. Then, for any numerical realization $(\tilde{A},\tilde{B})$
		, we have that $\emph{rank}(C_{\mathcal{T}}Q(\tilde{A},\tilde{B}))\le \left|\mathcal{N(X_T)}\right|$.}
\end{lem}
	\begin{proof}[Proof of Lemma~\ref{nec}]
	Suppose we have a numerical realization $(\tilde{A},\tilde{B})$. By Cayley-Hamilton theorem, 
	\begin{equation}\label{necc}\small
	\begin{aligned}
	\textrm{rank}(C_{\mathcal{T}}[\tilde{B},\tilde{A} Q(\tilde{A},\tilde{B})])&=\textrm{rank}(C_{\mathcal{T}}[\tilde{B},\tilde{A}\tilde{B},\dots,\tilde{A}^{n-1}\tilde{B},\tilde{A}^n\tilde{B}])\\
	&=\textrm{rank}([C_{\mathcal{T}}Q(\tilde{A},\tilde{B}),C_{\mathcal{T}}\tilde{A}^n\tilde{B}])\\
	&=\textrm{rank}(C_{\mathcal{T}}Q(\tilde{A},\tilde{B})).
	\end{aligned}
	\end{equation}
In $\mathcal{D}(\bar{A},\bar{B})$, let $m_1,m_2$ be the number of input, state vertices in $\mathcal{N(X_T)}$, respectively. Then, (\ref{necc}) yields,
	\begin{equation*}\label{neccc}\small
	\begin{aligned}
	\textrm{rank}(C_{\mathcal{T}}Q(\tilde{A},\tilde{B}))&=\textrm{rank}(C_{\mathcal{T}}[\tilde{B},\tilde{A}Q(\tilde{A},\tilde{B})])\\
	&\le\textrm{rank}(C_{\mathcal{T}}\tilde{B})+\textrm{rank}(C_{\mathcal{T}}\tilde{A}Q(\tilde{A},\tilde{B}))\\
	&\le m_1+\min (\textrm{rank}(C_{\mathcal{T}}\tilde{A}),\ \textrm{rank}(Q(\tilde{A},\tilde{B})))\\
	&\le m_1+m_2\\
	&=|\mathcal{N(X_T)}|.
	\end{aligned}
	\end{equation*}
	This completes the proof.
	\end{proof}
\begin{proof}[Proof of Theorem~\ref{theo3}] To show the necessity of the theorem, suppose that there exists a vertex $x_{i} \in \mathcal{X}$ that is not input-reachable, then the $i$-th row of controllability matrix will be zero row, which implies that $\textrm{rank}(Q(\tilde{A},\tilde{B}))<n$, for any numerical realization of the pair $(\bar{A},\bar{B})$. On the other hand, suppose there exists a set $\mathcal{S}\subseteq \mathcal{X}$, such that $\left|\mathcal{N(S)}\right|<\left|\mathcal{S}\right|$, then by Lemma~\ref{nec}, $\textrm{rank}(Q(\tilde{A},\tilde{B}))<n$, for any numerical realization of the pair $(\bar{A},\bar{B})$. Hence, the necessity is proved.
	
	To show the sufficiency, we proceed as follows. First, since $\left|\mathcal{N}(S)\right|\ge \left|\mathcal{S}\right|$, $\forall \mathcal{S}\subseteq\mathcal{X}$, it follows from Corollary~\ref{coro} that $\textrm{g--rank}([\bar{A},\bar{B}])=n$. Because all the state vertices are input-reachable, $(\bar{A},\bar{B})$ is irreducible. If we denote the term-rank of $\bar{A}$ as $k$, then by Lemma~\ref{theo2}, {there exists a proper variety $V\subset\mathbb{R}^{n_{\bar{A}}+n_{\bar{B}}}$ such that, if  $[\mathbf{p}_{\tilde{A}},\mathbf{p}_{\tilde{B}}]\in V^c$ then $\tilde{A}$ has $k$ nonzero, simple and controllable modes.} Let $\lambda$ be an eigenvalue of $\tilde{A}.$ On one hand, if $\lambda \ne 0$, then $\lambda$ is controllable by Lemma~\ref{theo2}. On the other hand, if $\lambda=0$, since $\textrm{g--rank}([\bar{A},\bar{B}])=n,$ then there exists a proper variety $W\subset\mathbb{R}^{n_{\bar{A}}+n_{\bar{B}}}$, such that if $[\mathbf{p}_{\tilde{A}},\mathbf{p}_{\tilde{B}}]\in W^c\cap V^c$, then $\textrm{rank}([\tilde{A},\tilde{B}])=n.$ As a result, $\lambda=0$ is controllable by the eigenvalue PBH test. Since all the modes of $\tilde{A}$ are controllable generically, $(\bar{A},\bar{B})$ is structurally controllable.
\end{proof}

\subsection{Proof of Theorem~\ref{theo4}}
\begin{proof}
	The necessity of Conditions \emph{1)} and \emph{2)} can be proved in a similar approach as the proof in Theorem~\ref{theo3}. What remains to be shown is their sufficiency. {It suffices to show that Conditions \emph{1)} and \emph{2)} result in that generically the left null space of target controllability matrix is trivial. }
	
	Suppose there exists an input-unreachable state vertex ${x_i}\in\mathcal{X}\setminus\mathcal{X_T}$. Since all the vertices in $\mathcal{X_T}$ are input-reachable, for $\forall {x_j}\in\mathcal{X_T}$, there is no path from ${x_j}$ to ${x_i}$, and there is also no path from ${x_i}$ to ${x_j}$ due to the symmetry in $\mathcal{D}(\bar{A})$. This implies in model (\ref{eq:1}) that the $i$th state has no impact on the dynamics of $\mathcal{T}$ corresponding states. Omitting the $i$th state from the system will not change the dynamics of $\mathcal{T}$ corresponding states. Hence, we could assume that $(\bar{A},\bar{B})$ is irreducible. By Lemma~\ref{theo2}, there exists a proper variety $V\subset\mathbb{R}^{n_{\bar{A}}+n_{\bar{B}}}$, such that if $[\mathbf{p}_{\tilde{A}},\mathbf{p}_{\tilde{B}}]\in V^c$, then all the nonzero modes of $\tilde{A}$ are controllable. In the rest of the proof, we assume $[\mathbf{p}_{\tilde{A}},\mathbf{p}_{\tilde{B}}]\in V^c$. Denote by $e_1,\dots,e_l$ the left eigenvectors corresponding to zero modes of $\tilde{A}$, and $e_{l+1},\dots,e_n$ the left eigenvectors for nonzero modes. Denote the left null space of a matrix $M$ as $\boldsymbol{N}(M^\top)$.
	
	From Lemma~\ref{theo2}, we have that if $[\mathbf{p}_{\tilde{A}},\mathbf{p}_{\tilde{B}}]\in V^c$, then $\boldsymbol{N}((Q(\tilde{A},\tilde{B}))^\top)\subseteq \textrm{span}\{e_1^\top,\dots,e_l^\top\}$. For the target set $\mathcal{T}$, define the matrix $C_{\mathcal{T}}$ according to (\ref{CT}). By the assumption $\left|\mathcal{N(S)}\right|\ge\left|\mathcal{S}\right|$, $\forall \mathcal{S}\subseteq \mathcal{X_T}$, and Corollary~\ref{coro}, we have that $\textrm{g--rank}(C_{\mathcal{T}}[\bar{A},\bar{B}])=\left|\mathcal{T}\right|$, which implies that there exists a proper variety $W\subset\mathbb{R}^{n_{\bar{A}}+n_{\bar{B}}}$, such that if $[\mathbf{p}_{\tilde{A}},\mathbf{p}_{\tilde{B}}]\in V^c\cap W^c$, then $\textrm{rank}(C_{\mathcal{T}}[\tilde{A},\tilde{B}])=|\mathcal{T}|$, i.e., $\boldsymbol{N}((C_{\mathcal{T}}[\tilde{A},\tilde{B}])^\top)=\mathbf{0}$. Define $\hat{I}\in\mathbb{R}^{n\times n}$ as 
	\begin{equation}\small\label{identity}
	[\hat{I}]_{ij}= 
	\begin{cases}
	1, & \text{ if } j=i,\ i\in\mathcal{T},\\
	0, & \text{otherwise.}
	\end{cases}
	\end{equation}
	We claim that there does not exist a nontrivial vector $e\in\mathbb{C}^n$ such that $\hat{I}e=e$, $e^\top\tilde{A}=0e^\top$ and $e^\top\tilde{B}=\mathbf{0}$. Otherwise, $e^\top[\tilde{A},\tilde{B}]=\mathbf{0}$, which contradicts $\boldsymbol{N}((C_{\mathcal{T}}[\tilde{A},\tilde{B}])^\top )=\mathbf{0}$.
	
	Hence, if $[\mathbf{p}_{\tilde{A}},\mathbf{p}_{\tilde{B}}]\in V^c\cap W^c$, then there is no nontrivial vector $v\in\mathbb{C}^{|\mathcal{T}|}$, such that $v^\top C_{\mathcal{T}}\in\textrm{span}\{e^\top_1,\dots,e^\top_l\}$. 
	Thus, generically, $\boldsymbol{N}((C_{\mathcal{T}}Q(\tilde{A},\tilde{B}))^\top)=\mathbf{0}$. The $(\bar{A},\bar{B})$ is structurally target controllable with respect to~$\mathcal{T}$.
\end{proof}


\bibliographystyle{IEEEtran}
{\small\bibliography{reference}}	
\end{document}